\documentclass[12pt]{amsart}

\usepackage{amsmath,amsthm,amssymb,enumitem,verbatim,stmaryrd,xcolor,microtype,graphicx,aliascnt}
\usepackage[T1]{fontenc}
\usepackage[utf8]{inputenc}
\usepackage[english]{babel} % allow for hyphenation for words with dash: e.g. ``non-archimedean'' does't break, but ``non\hyph archimedean'' does
\usepackage[top=3.5cm,bottom=3.5cm,left=3.5cm,right=3.5cm]{geometry}
\usepackage[bookmarksdepth=2,linktoc=page,colorlinks,linkcolor={red!70!black},citecolor={red!80!black},urlcolor={blue!80!black}]{hyperref}
\usepackage{tikz}\usetikzlibrary{matrix,arrows,decorations.markings,shapes}
\usepackage{tikz-cd}
\usepackage[mathcal]{euscript}                               % use the euscript caligraphic font with \mathscr{...}

\usepackage{mathptmx}                                        % times font
\usepackage[colorinlistoftodos,bordercolor=orange,backgroundcolor=orange!20,linecolor=orange,textsize=footnotesize]{todonotes} \makeatletter \providecommand \@dotsep{5} \def\listtodoname{List of Todos} \def\listoftodos{\@starttoc{tdo}\listtodoname} \makeatother %\Todo{} for margin notes, suppress in pdf with option [disable]

\allowdisplaybreaks % allows equation environments to break between the pages

\usepackage{etoolbox}
\makeatletter
\patchcmd{\@startsection}{\@afterindenttrue}{\@afterindentfalse}{}{}             %omit indentation of the first paragraph of a section
\patchcmd{\part}{\bfseries}{\bfseries\LARGE}{}{}
\patchcmd{\section}{\scshape}{\bfseries}{}{}\renewcommand{\@secnumfont}{\bfseries} %boldface no smallcaps section and subsection titles with numbers
\patchcmd{\@settitle}{\uppercasenonmath\@title}{\large}{}{}
\patchcmd{\@setauthors}{\MakeUppercase}{}{}{}
\addto{\captionsenglish}{} %boldface no smallcaps Abstract
\addto{\captionsenglish}{} %boldface no smallcaps Figure
\addto{\captionsenglish}{} %boldface no smallcaps Table
\makeatother

\usepackage{fancyhdr}

\addto\extrasenglish{}  % Change \autoref output from ``subsection 1.1'' to ``section 1.1''
\theoremstyle{plain}
\newtheorem{thm}{Theorem}[section] % provides command \autoref{}, which produces citations like ``Theorem 1.1''.
\newaliascnt{lemma}{thm}\newtheorem{lemma}[lemma]{Lemma}\aliascntresetthe{lemma}
\newaliascnt{cor}{thm}\newtheorem{cor}[cor]{Corollary}\aliascntresetthe{cor}
\newaliascnt{prop}{thm}\newtheorem{prop}[prop]{Proposition}\aliascntresetthe{prop}
\newtheorem*{claim}{Claim}
\newtheorem*{thm*}{Theorem}
\newtheorem*{lem*}{Lemma}
\newtheorem*{cor*}{Corollary}

\theoremstyle{definition}
\newaliascnt{df}{thm}\newtheorem{df}[df]{Definition}\aliascntresetthe{df}
\newaliascnt{rem}{thm}\newtheorem{rem}[rem]{Remark}\aliascntresetthe{rem}
\newaliascnt{ex}{thm}\newtheorem{ex}[ex]{Example}\aliascntresetthe{ex}

\newtheorem*{df*}{Definition}
\newtheorem*{ex*}{Example}
\newtheorem*{rem*}{Remark}

\theoremstyle{remark}

\DeclareRobustCommand{\gobblefour}[5]{}    % Command \SkipTocEntry for suppressing a section title in TOC
\DeclareFontFamily{OT1}{pzc}{}                                % Script font for small caligraphic letter, like in \cMat
\DeclareFontShape{OT1}{pzc}{m}{it}{<-> s * [1.10] pzcmi7t}{}
\DeclareMathAlphabet{\mathpzc}{OT1}{pzc}{m}{it}
\DeclareSymbolFont{sfoperators}{OT1}{bch}{m}{n} \DeclareSymbolFontAlphabet{\mathsf}{sfoperators} \makeatletter\def\operator@font{\mathgroup\symsfoperators}\makeatother % different font for math operators
\DeclareSymbolFont{cmletters}{OML}{cmm}{m}{it}              
\DeclareSymbolFont{cmsymbols}{OMS}{cmsy}{m}{n}
\DeclareSymbolFont{cmlargesymbols}{OMX}{cmex}{m}{n}
\DeclareMathSymbol{\myjmath}{\mathord}{cmletters}{"7C}     \let\jmath\myjmath %Defining the missing commands: \jmath, \amalg and \coprod
\DeclareMathSymbol{\myamalg}{\mathbin}{cmsymbols}{"71}     
\DeclareMathSymbol{\mycoprod}{\mathop}{cmlargesymbols}{"60}
\DeclareMathSymbol{\myalpha}{\mathord}{cmletters}{"0B}     \let\alpha\myalpha %Greek letters from Computer Modern since the Greek letters from mathptmx are too large
\DeclareMathSymbol{\mybeta}{\mathord}{cmletters}{"0C}      \let\beta\mybeta
\DeclareMathSymbol{\mygamma}{\mathord}{cmletters}{"0D}     \let\gamma\mygamma
\DeclareMathSymbol{\mydelta}{\mathord}{cmletters}{"0E}     \let\delta\mydelta
\DeclareMathSymbol{\myepsilon}{\mathord}{cmletters}{"0F}   \let\epsilon\myepsilon
\DeclareMathSymbol{\myzeta}{\mathord}{cmletters}{"10}      \let\zeta\myzeta
\DeclareMathSymbol{\myeta}{\mathord}{cmletters}{"11}       \let\eta\myeta
\DeclareMathSymbol{\mytheta}{\mathord}{cmletters}{"12}     \let\theta\mytheta
\DeclareMathSymbol{\myiota}{\mathord}{cmletters}{"13}      \let\iota\myiota
\DeclareMathSymbol{\mykappa}{\mathord}{cmletters}{"14}     \let\kappa\mykappa
\DeclareMathSymbol{\mylambda}{\mathord}{cmletters}{"15}    \let\lambda\mylambda
\DeclareMathSymbol{\mymu}{\mathord}{cmletters}{"16}        \let\mu\mymu
\DeclareMathSymbol{\mynu}{\mathord}{cmletters}{"17}        \let\nu\mynu
\DeclareMathSymbol{\myxi}{\mathord}{cmletters}{"18}        \let\xi\myxi
\DeclareMathSymbol{\mypi}{\mathord}{cmletters}{"19}        \let\pi\mypi
\DeclareMathSymbol{\myrho}{\mathord}{cmletters}{"1A}       \let\rho\myrho
\DeclareMathSymbol{\mysigma}{\mathord}{cmletters}{"1B}     \let\sigma\mysigma
\DeclareMathSymbol{\mytau}{\mathord}{cmletters}{"1C}       \let\tau\mytau
\DeclareMathSymbol{\myupsilon}{\mathord}{cmletters}{"1D}   \let\upsilon\myupsilon
\DeclareMathSymbol{\myphi}{\mathord}{cmletters}{"1E}       \let\phi\myphi
\DeclareMathSymbol{\mychi}{\mathord}{cmletters}{"1F}       \let\chi\mychi
\DeclareMathSymbol{\mypsi}{\mathord}{cmletters}{"20}       \let\psi\mypsi
\DeclareMathSymbol{\myomega}{\mathord}{cmletters}{"21}     \let\omega\myomega
\DeclareMathSymbol{\myvarepsilon}{\mathord}{cmletters}{"22}\let\varepsilon\myvarepsilon
\DeclareMathSymbol{\myvartheta}{\mathord}{cmletters}{"23}  \let\vartheta\myvartheta
\DeclareMathSymbol{\myvarpi}{\mathord}{cmletters}{"24}     \let\varpi\myvarpi
\DeclareMathSymbol{\myvarrho}{\mathord}{cmletters}{"25}    \let\varrho\myvarrho
\DeclareMathSymbol{\myvarsigma}{\mathord}{cmletters}{"26}  \let\varsigma\myvarsigma
\DeclareMathSymbol{\myvarphi}{\mathord}{cmletters}{"27}    \let\varphi\myvarphi

\DeclareMathOperator{\hypersum}{\,\raisebox{-2.2pt}{\larger[2]{$\boxplus$}}\,}

\newcommand\N{{\mathbb N}}
\newcommand\NN{{\mathbb N}}

\renewcommand\P{{\mathbb P}}

\newcommand\R{{\mathbb R}}
\renewcommand\S{{\mathbb S}}

\newcommand\U{{\mathbb U}}

\newcommand\Z{{\mathbb Z}}

\newcommand\cP{{\mathcal P}}

\renewcommand\geq{\geqslant}
\renewcommand\leq{\leqslant}

\newcommand{\hyperplus}{\mathrel{\,\raisebox{-1.1pt}{\larger[-0]{$\boxplus$}}\,}}

\renewcommand\emptyset\varnothing

\newcommand\tz{\texorpdfstring{\textbf{Tianyi:} }{Tianyi: }}
\newcommand\mb{\texorpdfstring{\textbf{Matt:} }{Matt: }}

\title{Fusion rules for pastures and tracts}

\author{Matthew Baker}
\address{\rm School of Mathematics, Georgia Institute of Technology, Atlanta, USA}
\email{mbaker@math.gatech.edu}

\author{Tianyi Zhang}
\address{\rm School of Mathematics, Georgia Institute of Technology, Atlanta, USA}
\email{kafuka@gatech.edu}

\begin{document}

\begin{abstract}
Baker and Bowler defined a category of algebraic objects called tracts which generalize both partial fields and hyperfields. They also defined a notion of weak and strong matroids over a tract $F$, and proved that if $F$ is perfect, meaning that $F$-vectors and $F$-covectors are orthogonal for every matroid over $F$, then the notions of weak and strong $F$-matroids coincide. We define the class of strongly fused tracts and prove that such tracts are perfect. 
We also show that both partial fields and stringent hyperfields are strongly fused; in this way, our criterion for perfection generalizes results of Baker-Bowler and Bowler-Pendavingh. 
\end{abstract}

% Removed: We in fact prove a more general result which implies that given a tract $F$, there is a tract $\sigma (F)$ with the same 3-term additive relations as $F$ such that weak $F$-matroids coincide with strong $\sigma (F)$-matroids. 

\maketitle
{\ \vspace{-515pt}\\ \flushright\tiny\em\ \today\\}\vspace{480pt}

% \begin{small} \tableofcontents \end{small}

%%%%%%%%%%%%%%%%%%%%%%%%%%%%%%%%%%%%%%%%%%%%%%%%%%%%%%%%%%%%%%%%%%%%%%%%%%%%%%%%%%%%%%%%%%%%%%%%%%%%%%%%%%%%%%%%%%%%%%%%%%%%%%%%%%%%%%%%%%%%%%%%%%%%%%%%%%%%%%%%%%%%%
%%%%%%%%%%%%%%%%%%%%%%%%%%%%%%%%%%%%%%%%%%%%%%%%%%%%%%%%%%%%%%%%%%%%%%%%%%%%%%%%%%%%%%%%%%%%%%%%%%%%%%%%%%%%%%%%%%%%%%%%%%%%%%%%%%%%%%%%%%%%%%%%%%%%%%%%%%%%%%%%%%%%%

\section{Introduction}
\label{introduction}

\subsection{Overview}

In \cite{Baker-Bowler19}, Baker and Bowler define a category of algebraic objects called {\bf tracts} which generalize both partial fields and hyperfields (in particular, they generalize fields). Given a tract $F$, Baker and Bowler define the notions of {\bf weak} and {\bf strong} matroids over $F$, and they prove that if $F$ is {\bf perfect} (meaning that $F$-vectors and $F$-covectors are orthogonal for every $F$-matroid) then the notions of weak and strong $F$-matroids coincide. The authors also show that partial fields and doubly distributive hyperfields are always perfect.

\medskip

The fact that doubly distributive hyperfields are perfect was generalized by Bowler--Pendavingh \cite{Bowler-Pendavingh19} and Bowler--Su \cite{Bowler-Su20}, who showed that every {\bf stringent} hyperfield\footnote{In fact, Bowler and Pendavingh work in the more general context of not necessarily multiplicatively commutative {\bf skew hyperfields}, but for simplicity we restrict to the commutative case in this paper.} is perfect and every doubly distributive hyperfield is stringent. 

\medskip

In this paper, we define the class of {\bf strongly fused} tracts and prove that such tracts are perfect. We also show that both partial fields and stringent hyperfields are strongly fused, so our criterion for perfection generalizes results from \cite{Baker-Bowler19, Bowler-Pendavingh19,Bowler-Su20}. 

\medskip

The proof of our main theorem (strongly fused tracts are perfect) is heavily influenced by the paper \cite{Dress-Wenzel92} of Dress and Wenzel, though the details differ in a number of places.
% and we are able to reach similar conclusions with weaker hypotheses.
% \mb I removed the previous comment for now, until we're able to figure out how to place weak distributivity into the context of tracts.

\medskip

We now explain our results in more detail, deferring proofs of the main propositions and theorems to the later sections.

\subsection{Pastures and tracts}

Given an abelian group $G$, let $\NN[G]$ denote the group semiring associated to $G$.
% (with $0 \cdot g$ identified with the zero element of $\NN[G]$ for all $g \in G$). 
For $\alpha \in \NN[G]$, let $\| \alpha \|$ be the unique natural number $m$ such that $\alpha$ is a sum of $m$ elements of $G$ (with $\| 0 \| = 0$). Thus $\| \alpha \| = 1$ iff $\alpha \in G$, and we have $\| \alpha \cdot \beta \| = \| \alpha \| \cdot \| \beta \|$ and $\| \alpha + \beta \| = \| \alpha \| + \| \beta \|$ for all $\alpha,\beta \in \NN[G]$. For $m \in \NN$, let $\NN[G]_{\leq m}$ denote the set of all $\alpha \in \NN[G]$ with $\| \alpha \| \leq m$.

\begin{df}
A {\bf tract} (resp. {\bf pasture}) is a multiplicatively written commutative monoid $F$ with an absorbing element $0$ such that $F^\times := F \setminus \{ 0 \}$ is a group, together with a subset $N_F$ of $\NN[F^\times]$ (resp. $\NN[F^\times]_{\leq 3}$) satisfying:
\begin{itemize}
\item[(T1)] The zero element of $\NN[F^\times]$ belongs to $N_F$.
\item[(T2)] There is a unique element $\epsilon$ of $F^\times$ with $1 + \epsilon \in N_F$.
\item[(T3)] $N_F$ is closed under the natural action of $F^\times$ on $\NN[F^\times]$.
\end{itemize}
\end{df}  

We call $N_F$ the {\bf null set} of $F$.

We will write $-1$ instead of $\epsilon$, $-x$ instead of $\epsilon x$, and $x - y$ instead of $x + \epsilon y$ for $x,y \in F$. 
(Note, however, that $x - x$ is not equal to $0$ in $\NN[F^\times]$, we merely have $x - x \in N_F$.)

% We will also write $\NN[F^\times]$ instead of $\NN[F^\times]$.

\medskip

Roughly speaking, tracts are the natural algebraic setting for considerations involving strong matroids and pastures are the natural setting for considerations involving weak matroids. Since we're interested in conditions such as perfection which guarantee that every weak matroid is strong, it is natural to explore the relationship between pastures and tracts.

\medskip

\begin{df}
A {\bf morphism} of tracts (or pastures) is a map $\varphi : F \to F'$ such that $\varphi(0)=0$, $\varphi$ induces a group homomorphism from $F^\times$ to $(F')^\times$, and $\varphi(N_F) \subseteq N_{F'}$.
\end{df}

\subsection{Partial fields and hyperfields}

Given a pair $(G,R)$ consisting of a commutative ring $R$ with $1$ and a subgroup $G \leq R^\times$ containing $-1$, we can associate a pasture $P = P_{G,R}$ with $(G,R)$ by setting $P^\times = G$ and declaring that $x+y+z \in N_P$ if and only if $x+y+z=0$ in $P$.
Pastures of this form are called {\bf partial fields}.

\medskip

Roughly speaking, a {\bf hyperfield} is an algebraic structure which behaves like a field except that addition is allowed to be multivalued. More precisely, a hyperfield $H$ consists of a multiplicative monoid with an absorbing element $0$ such that $H^\times = H \setminus \{ 0 \}$ is an abelian group, an involution $x \mapsto -x$ fixing $0$, and a commutative {\bf hyperoperation} which associates to each pair of elements $a,b \in H$ a non-empty subset $a \boxplus b$ of $H$.
The multiplication and hyperaddition are required to satisfy a number of axioms including commutativity and distributivity, and we require for each $a,b \in H$ that $0 \in a \boxplus b$ if and only if $a = -b$. There is also a {\bf reversibility} axiom which says that $c\in a\hyperplus b$ if and only if 
$b \in c \boxplus (-a)$.
% Previously: $(-a)\in (-c)\hyperplus b$.

\begin{df}
The tract $F_H$ (resp. pasture $P_H$) associated to a hyperfield $H$ has multiplicative group $H^\times$ and null set defined by $\sum_{i=1}^k x_i \in N_H$ iff $0 \in \boxplus_{i=1}^k x_i$ (resp. $x+y+z \in N_H$ if and only if $0 \in x\boxplus y \boxplus z$).
\end{df}

% Unless otherwise specified, we will always identify a hyperfield $H$ with its associated tract $F_H$. 

\medskip

If $P$ is a pasture and we set $x \hyperplus y = \{ z \in P \; : \; x+y -z \in N_P \}$, the pasture $P$ corresponds to a field if and only if $\hyperplus$ is an associative binary operation. 
Moreover, $P=P_H$ for some hyperfield $H$ if and only if $x \hyperplus y$ contains {\bf at least one} element for all $x,y \in P$ and $\hyperplus$ is associative (in the sense of set-wise addition),
and $P=P_{(G,R)}$ for some partial field $(G,R)$ if and only if $x \hyperplus y$ contains {\bf at most one} element for all $x,y \in P$ and satisfies a suitable associative law (which is a bit complicated to state, cf.~\cite[Section 2.2]{Pendavingh-vanZwam10b}).
Pastures thus generalize (and simplify) both hyperfields and partial fields by imposing no conditions on the size of the sets $x \hyperplus y$ and no associativity conditions.

\begin{df}
A hyperfield $H$ is {\bf stringent} if $|a \boxplus b| = 1$ for all $a,b \in H$ with $a \neq -b$.
\end{df}

\begin{df}
A hyperfield $H$ is {\bf doubly-distributive} if $(a\hyperplus b)(c\hyperplus d) \ = \ ac \hyperplus bc \hyperplus ad \hyperplus bd$ for all $a,b,c,d\in H$.
\end{df}

Here are a few examples of hyperfields and their associated tracts:

\begin{ex}[Sign hyperfield]
The sign hyperfield $\S$ consists of the multiplicative monoid $\{0,\pm 1\}$, together with the hyperaddition rule given by $1\hyperplus1 = 1$, $(-1)\hyperplus(-1) = -1$, and $(-1)\hyperplus(-1) = \{-1,0,1\}$. 
As a tract, $N_\S$ consists of 0 and all formal sums $\sum x_i$ with at least one $1$ and one $-1$. 
\end{ex}

\begin{ex}[$\S \times \S$]
Products exist in both the category of hyperfields and the category of tracts. As a multiplicative monoid, $\S\times \S$ is given by the Cartesian product of $\{0,\pm 1\}$ with itself, while $N_{\S\times\S}$ consists of 0 and all formal sums $\sum (x_i,y_i)$ such that both $\sum x_i$ and $\sum y_i$ are in $N_\S$.
\end{ex}

\begin{ex}[Phase hyperfield]
The phase hyperfield $\P$ consists of the multiplicative monoid $\{0\}\cup \S^1$, where $\S^1$ is the complex unit circle, together with the following hyperaddition rule. Given $x_i \in \P$, the hypersum $\hypersum_{i=1}^n x_i$ is the set of phases of all complex numbers in the cone 
$\sum_{i=1}^n c_ix_i$ with $c_i\in \R_{>0}$. As a tract, $\sum_{i=1}^n x_i \in N_\P$ if and only if there exist $c_1,\ldots, c_n \in \R_{>0}$ such that $\sum_{i=1}^n c_i\cdot x_i = 0$ in ${\mathbb C}$.  
\end{ex}

\subsection{The fusion axiom}

% Although one can restrict the null set of a tract to just the terms of degree $\leq 3$ and get a pasture, and extend a pasture to a tract via the inclusion $\NN[F^\times]_{\leq 3} \subset \NN[F^\times]$, these operations are not very useful in practice.
Although one can trivially extend a pasture to a tract via the inclusion $\NN[F^\times]_{\leq 3} \subset \NN[F^\times]$, this way of viewing pastures as tracts is not very useful in practice.
Instead, it is more useful to define the tract associated to a pasture by inductively ``fusing'' together additive relations of smaller degree to generate higher-degree relations.
More precisely, consider the following {\bf fusion axiom}:
\begin{itemize}
\item[(F)] If $\alpha + z$ and $\beta - z$ are in $N_F$ with $\alpha, \beta \in \NN[F^\times]$ and $z \in F$, then $\alpha + \beta \in N_F$.
\end{itemize}  

Given a pasture $P$, let $\overline{P}$ be the tract whose multiplicative group is $P^\times$ and whose null set is the smallest subset of $\NN[P^\times]$ containing $N_P$ and satisfying the fusion axiom.

The proof of the following result is left as an exercise:

\begin{prop}
\label{prop:FFPT}
The map $P \mapsto \overline{P}$ defines a fully faithful functor from pastures to tracts. A tract $F$ is equal to $\overline{P}$ for some pasture $P$ if and only if $F$ satisfies the fusion axiom (F) and every $\gamma \in N_F \cap \N[F^\times]_{\geq 4}$ can be written as $\alpha + \beta$ for some $\alpha,\beta \in \N[F^\times]_{\geq 2} $ and $z \in F$ with $\alpha + z$ and $\beta - z$ in $N_F$.
\end{prop}

In particular, there is no harm in identifying a pasture $P$ with the corresponding tract $\overline{P}$. 

\medskip

Note that the tract $\overline{P}$ associated to a pasture $P$ is in fact an {\bf idyll} (cf.~ \cite[Section 1.2.2]{Baker-Lorscheid18}), meaning that $N_F$ is an ideal in the semiring $\NN[F^\times]$; this is equivalent to requiring:
\begin{itemize}
\item[(I)] If $\alpha, \beta \in N_F$ then $\alpha + \beta \in N_F$.
\end{itemize}  

\medskip

For hyperfields, we have the following pleasant correspondence (which was in fact our motivation for the fusion axiom):

\begin{prop}
\label{prop:HAP}
If $H$ is a hyperfield and $F_H$ (resp. $P_H$) is the associated tract (resp. pasture) then $F_H = \overline{P_H}$.
\end{prop}

For later reference, we also define a functor from tracts to pastures: given a tract $F$, define  
% $\tau_{\leq 3}(F)$ 
the {\bf 3-term truncation of $F$} to be the pasture whose multiplicative group is $F^\times$ and whose null set is $N_F \cap \NN[F^\times]_{\leq 3}$.

\subsection{The strong fusion axiom}

The present paper is motivated by the observation that many tracts of interest, such as partial fields and stringent hyperfields, satisfy a property that is stronger than (F) and which turns out to be sufficient to guarantee perfection.

\medskip

More precisely, consider the following {\bf strong fusion axiom}:
\begin{itemize}
\item[(SF)] If $\alpha + \gamma$ and $\beta - \gamma$ are in $N_F$ with $\alpha, \beta, \gamma \in \NN[F^\times]$ and either $\gamma = 0$ or $\gamma \not\in N_F$, then $\alpha + \beta \in N_F$.
\end{itemize}  

Note that the fusion axiom is precisely the case where $\gamma \in F$, and in particular a tract satisfying (SF) (which we call a {\bf strongly fused} tract) automatically satisfies (F).

\medskip

The main result of this paper is:

\begin{thm} \label{thm:mainthm}
Every strongly fused tract is perfect.
\end{thm}

In fact, we will prove a stronger version of Theorem~\ref{thm:mainthm} in which we replace (SF) with the modified axiom:

\begin{itemize}
\item[(MSF)] If $\alpha + \gamma$ and $\beta - \gamma$ are in $N_F$ with $\alpha, \beta, \gamma \in \NN[F^\times]$ and either $\gamma = 0$ or $\gamma \not\in N_F$, and if $\| \alpha + \beta \| \geq 4$, then $\alpha + \beta \in N_F$.
\end{itemize}  

Our proof will show more generally that a tract satisfying (MSF) is {\bf strongly perfect}, a notion which will be defined in Section~\ref{sec:MOTSPT} (but it turns out to be equivalent to perfection in the usual sense; see Theorem~\ref{thm:perfectequalsstronglyperfect}).

\medskip

As a corollary of the strengthened version of Theorem~\ref{thm:mainthm}, we obtain:

\begin{cor}
There is a rule $F \mapsto \sigma(F)$ which associates to each tract $F$ a strongly perfect tract $\sigma(F)$ and which is the identity map on tracts satisfying (MSF).
% Old version: If $P$ is a pasture, there is a (strongly) perfect tract $\sigma (P)$ whose 3-term truncation is $P$. In particular, there is a bijection between weak $P$-matroids and strong $\sigma (P)$-matroids.
\end{cor}

\begin{proof}
We can take $\sigma (F)$ to be the tract whose multiplicative group is $F^\times$ and whose null set is defined as follows.
Let $N^{(1)}=N_F$, and for $k \geq 2$ define $N^{(k)}$ to be the set of all elements of the form $\alpha + \beta$ with $\alpha, \beta \in N^{(k-1)}$ or $\alpha + \gamma$ and $\beta - \gamma$ in $N^{(k-1)}$ for some $\gamma \not\in N^{(k-1)}$, and such that $\| \alpha + \beta \| \geq 4$. 
Set $N_{\sigma (F)} := \bigcup_{k \geq 1} N^{(k)}$. It is easy to see that $\sigma(F)=F$ if $F$ satisfies (MSF). We claim that $\sigma(F)$ satisfies (MSF) for every tract $F$.
Indeed, suppose $\alpha + \gamma$ and $\beta - \gamma$ are in $N_{\sigma (P)}$ with either $\gamma = 0$ or $\gamma \not\in N_{\sigma (P)}$, and assume furthermore that $\| \alpha + \beta \| \geq 4$. Then by definition there exists $k\geq 1$ such that $\alpha + \gamma$ and $\beta - \gamma$ are in $N^{(k-1)}$, and if $\gamma \not\in N_{\sigma (P)}$ then $\gamma \not\in N^{(k-1)}$ since $N^{(k-1)} \subset N_{\sigma (P)}$. By the definition of $N^{(k)}$ we have $\alpha + \beta \in N^{(k)}$, hence $\alpha + \beta \in N_{\sigma (P)}$.
\end{proof}

\subsection{Stringent hyperfields and the strong fusion axiom}

It is easy to see that the tract embedding of a partial field satisfies the strong fusion axiom. 
% see Lemma~\ref{lem:FPSF} below. 
For hyperfields, we show:

\begin{prop}
\label{prop:SFforHyperfields}
If $H$ is a hyperfield, then $H$ satisfies the strong fusion axiom if and only if $H$ is stringent.
\end{prop}

In particular, this gives a new proof of the fact, originally proved by Bowler and Pendavingh in \cite{Bowler-Pendavingh19}, that stringent hyperfields are perfect.

\begin{rem}
If we removed the assumption that $\gamma \not\in N_F$ in (SF), then stringent hyperfields would no longer satisfy this property. 
For example, in the sign hyperfield ${\mathbb S}$ with $\alpha = 1$, $\beta = 1$, and $\gamma = 1 + (-1)$, we have $\alpha + \gamma, \beta - \gamma \in N_{\mathbb S}$ but $\alpha + \beta \notin N_{\mathbb S}$.
\end{rem}

It would be useful to have a natural and easily verified sufficient condition which implies perfection, is satisfied by stringent hyperfields and partial fields, and which is stable under taking finite products (since one easily shows that the product of perfect pastures is perfect.) 
Unfortunately, neither (SF) nor (MSF) is stable under products, as the following shows:

% Also mention that it would be nice to have an explicitly describable and functorial ``perfection'' operator.

% Taking product does not preserve (SF) or (MSF). \par
\begin{ex}
A counterexample which applies to both (SF) and (MSF) is $F = \mathbb{S}\times \mathbb{S}$, where $\mathbb{S}$ is the sign hyperfield. 
Indeed, note that if $\gamma = (1,1)+(-1,1)$, $\alpha = (1,-1)+(1,-1)$, and $\beta = (1,1)+(1,1)$ then $\alpha + \gamma, \beta - \gamma \in N_F$ and $\gamma \not\in N_F$ but $\alpha + \beta \not\in N_F$.
% $(1,-1)+(1,1)+(-1,1)\in N_{F}$ and $(-1,-1)+(1,1)+(-1,1)\in N_{F}$ but $(1,-1)+(1,1)\notin N_{F}$. The reason is that $(1,1)+(-1,1)$ is not a null set for the product but if we look at just the first component $1+(-1)$ is in the null set of the first component so we don't have the strong fusion axiom. 
\end{ex}
% One can show by definition that the product of two perfect pastures remains perfect, so this implies (MS) of (MSF) is not sufficient. 

% We also show:

% \begin{prop}
% \label{prop:DDimpliesSF}
% If $H$ is a doubly distributive hyperfield then $H$ satisfies the strong fusion axiom.
% \end{prop}

% Combining the two previous propositions gives a simple proof the fact, originally proved by Bowler and Pendavingh in \cite{Bowler-Pendavingh19} using the classification theorem of \cite{Bowler-Su20}, that doubly distributive hyperfields are stringent. Our proof does not make use of the Bowler--Su classification theorem and is therefore significantly more elementary.

\subsection{Structure of the paper} 
The proofs of Propositions~\ref{prop:HAP} and~\ref{prop:SFforHyperfields} are given in Section~\ref{sec:HPFPT}. In Section~\ref{sec:MOTSPT} we recall the definition of an $F$-matroid and define what it means for an $F$-matroid (resp. a tract) to be strongly perfect. We then prove that strong perfection and perfection coincide. The proof of (a strengthening of) Theorem~\ref{thm:mainthm} is given in Section~\ref{sec:SFTSP}. Finally, in Section~\ref{sec:comparison} we compare our results to those of Dress--Wenzel.

\subsection{Acknowledgments}
We thank Oliver Lorscheid and Nathan Bowler for helpful conversations. We also thank the anonymous referees for their helpful comments; in particular, one of the referees suggested the statement and proof of Theorem~\ref{thm:perfectequalsstronglyperfect}.
The first author was supported in part by a Simons Foundation Collaboration Grant.

\section{Hyperfields, partial fields, and fusion axioms}
\label{sec:HPFPT}

In this section we prove Propositions~\ref{prop:HAP} and~\ref{prop:SFforHyperfields}.

\subsection{Hyperfields, partial fields, and the fusion axiom}
Our goal in this section is to prove Proposition~\ref{prop:HAP}. 
% and show that the analogous statement for partial fields is false.
In order to do this, we first recall the precise definition of a hyperfield.

\begin{df}
A {\bf commutative hypergroup} is a set $G$ together with a distinctive element $0$ and a {\bf hyperaddition}, which is a map
\[
 \hyperplus: \quad G\times G \quad \longrightarrow \quad \cP(G)
\]
into the power set $\cP(G)$ of $G$, such that:
\begin{enumerate}[label={(HG\arabic*)}]
 \item\label{HG1} $a\hyperplus b$ is not empty,    \hfill\textit{(nonempty sums)}
 % \item\label{HG2} $\big\{\, a\hyperplus d\, \big| \, d\in b\hyperplus c \, \big\} = \big\{\, d\hyperplus c\, \big| \, d\in a\hyperplus b \, \big\}$, \hfill\textit{(associativity)}
 \item\label{HG2} $\bigcup_{d\in b\hyperplus c} a\hyperplus d = \bigcup_{d\in a\hyperplus b} d\hyperplus c$, \hfill\textit{(associativity)}
 \item\label{HG3} $0\hyperplus a=a\hyperplus 0=\{a\}$,  \hfill\textit{(neutral element)}
 \item\label{HG4} there is a unique element $-a$ in $G$ such that $0\in a\hyperplus (-a)$,   \hfill\textit{(inverses)}
 \item\label{HG5} $a\hyperplus b=b\hyperplus a$,    \hfill\textit{(commutativity)}
 \item\label{HG6} $c\in a\hyperplus b$ if and only if $b \in c \hyperplus (-a)$    \hfill\textit{(reversibility)}
\end{enumerate}
for all $a,b,c\in G$. 
\end{df}

Thanks to commutativity and associativity, it makes sense to define hypersums of several elements $a_1,\dotsc, a_n$ unambiguously by the recursive formula

\[
 \hypersum_{i=1}^n \ a_i \ = \  \bigcup_{b\in\hypersum_{i=1}^{n-1} a_i} b\hyperplus a_n.
\]

\begin{df}
A {\bf (commutative) hyperring} is a set $R$ together with distinctive elements $0$ and $1$ and with maps $\hyperplus:R\times R\to\cP(R)$ and $\cdot:R\times R\to R$ such that 
\begin{enumerate}[label={(HR\arabic*)}]
 \item\label{HR1} $(R,\hyperplus,0)$ is a commutative hypergroup,
 \item\label{HR2} $(R,\cdot,1)$ is a commutative monoid,
 \item\label{HR3} $0\cdot a=a\cdot 0=0$,
 \item\label{HR4} $a\cdot(b\hyperplus c)=ab\hyperplus ac$
\end{enumerate}
for all $a,b,c\in R$ where $a\cdot(b\hyperplus c)=\{ad \; |\; d\in b\hyperplus c\}$. 

A {\bf hyperfield} is a hyperring $H$ such that $0\neq1$ and every nonzero element has a multiplicative inverse, i.e., $H^\times = H \setminus \{ 0 \}$.
\end{df}

% \subsection{Partial fields}

\begin{lemma} \label{lem:HF}
If $H$ is a hyperfield and $F_H$ is the associated tract, then $F_H$ satisfies the fusion axiom.
\end{lemma}

\begin{proof}
Suppose $\alpha + z$ and $\beta - z$ are in $N_H$, where $\alpha = \sum_{i=1}^k x_i$, $\beta = \sum_{j=1}^l y_j$, and $z \in H$. Then by definition, $0 \in (\hypersum_{i=1}^k x_i) \hypersum z $ and $0 \in (\hypersum_{j=1}^l y_j) \hypersum -z$. By the inverse axiom (HG4), $-z$ is in $\hypersum_{i=1}^k x_i$ and $z$ is in $\hypersum_{j=1}^l y_j $. Hence, 0 is in $(\hypersum_{i=1}^k x_i) \hypersum (\hypersum_{j=1}^l y_j )$, i.e., $\alpha + \beta \in N_H$.
\end{proof}

We can now prove Proposition \ref{prop:HAP}.

\begin{proof}[Proof of Proposition \ref{prop:HAP}]
First, notice that $F := F_H$, $P := P_H$, and $\overline{P_H}$ all have the same multiplicative group, so it suffices to prove that $F_H$ and $\overline{P_H}$ have the same null set. 

\par

Next, note that the null set of $P_H$ is contained in the null set of $F_H$, so by Lemma \ref{lem:HF} and the definition of $\overline{P_H}$, $N_{\overline{P}} \subseteq N_F$.  
Conversely, we will prove by induction on $\|\gamma\|$ that if $\gamma \in N_F$ then $\gamma \in N_{\overline{P}}$. 

\par

The base case $\| \gamma \| \leq 3$ is trivial. Assume, then, that $\gamma \in N_{\overline{P}}$ for every $\gamma \in N_F$ with $\|\gamma\|< k $, and let $\gamma \in N_F$ have norm $k \geq 4$. Then $\gamma = \sum_{i=1}^k x_i$ with $x_i \in H$, and since $\gamma \in N_F$ we have $0 \in \hypersum_{i=1}^k x_i$ in $H$. By the associativity axiom (HG2), we have $0\in (\hypersum_{i=1}^{k-2} x_i) \hypersum (x_{k-1}\hypersum x_k)$. Since $k \geq 4$, there exists some element $z\in H$ such that $z \in \hypersum_{i=1}^{k-2} x_i$ and $-z \in x_{k-1}\hypersum x_k$. By the reversibility axiom (HG4), $0 \in -z \hypersum_{i=1}^{k-2} x_i$ and $0\in z \hypersum x_{k-1}\hypersum x_k$, which means, setting $\alpha = \sum_{i=1}^{k-2} x_i$ and $\beta = x_{k-1} + x_k$, that $\alpha - z$ and $\beta + z$ both belong to $N_F$.
By the inductive hypothesis, these two elements of $N_F$ are in $N_{\overline{P}}$. Applying the fusion axiom gives $\gamma = \alpha + \beta \in N_{\overline{P}}$.
\end{proof}

\subsection{Stringent hyperfields and the strong fusion axiom}
\label{sec:SHDDHSF}

Our goal in this section is to prove Proposition~\ref{prop:SFforHyperfields}. The following is a more precise version of this result.

\begin{prop}
Let $H$ be a hyperfield. Then the following are equivalent:\par 
(1) $H$ is stringent.\par 
(2) If $x_1, ... , x_k \in H$ and $0 \not\in \hypersum_{i=1}^k x_i$ then $|\hypersum_{i=1}^k x_i|=1$.\par
(3) The tract $F=F_H$ associated to $H$ satisfies the the strong fusion axiom (SF). 
\end{prop}

\begin{proof}
$(1) \Rightarrow (2)$:
This follows from \cite[Lemma 39]{Bowler-Pendavingh19}.  \par
$(2)\Rightarrow (3)$: 
Suppose $\alpha = \sum_{i=1}^\ell x_i, \beta = \sum_{j=1}^m y_j$, and $\gamma=\sum_{k=1}^n z_k$ satisfy 
$\alpha + \gamma, \beta - \gamma \in N_F$ and $\gamma = 0$ or $\gamma \not\in N_F$.
If $\gamma = 0$, the result follows from Lemma~\ref{lem:HF}. So we may assume that $\gamma \not\in N_F$.
% Since $F$ is the tract associated to $H$, $\alpha,\beta,\gamma$ correspond to hypersums $\hypersum_{i=1}^m x_i, \hypersum_{j=1}^n y_j$, and $\hypersum_{i=1}^m z_i$, respectively. 
% Since $\gamma \not\in N_F$, (2) $\gamma$ is essentially single valued and $\gamma = z$. Then, 
Thus $\hypersum_{k=1}^n z_k = \{ z \}$ is a singleton and we have $-z\in \hypersum_{i=1}^\ell x_i$ and $z\in \hypersum_{j=1}^m y_j$.
It follows that $0\in (\hypersum_{i=1}^\ell x_i) \hypersum (\hypersum_{j=1}^m y_j)$, i.e. $\alpha + \beta \in N_F.$\par
$(3) \Rightarrow (1)$:
Suppose there exist $x,y\in H$ such that $x\neq -y$ and $|x\hypersum y|\geq 2$. Choosing $z\neq z' \in x\hypersum y$, we have $0\in -z\hypersum x \hypersum y$ and $0\in -z'\hypersum x \hypersum y$. Since $x\neq -y$ we have $0\notin x\hypersum y $, and thus (SF) implies that 
$0\in z\hypersum -z'$, which contradicts the fact that $z\neq z'$.
\end{proof}

\section{Matroids over tracts and strongly perfect tracts}
\label{sec:MOTSPT}

Our goal in this section is to define what it means for a tract $F$ to be strongly perfect, and to show that $F$ is strongly perfect if and only if it is perfect.
% , as preparation for the statement and proof of our main theorem.
To do this, we need to introduce some terminology related to matroids over tracts.

\begin{df}
(Involution) Let $F$ be a tract. An {\bf involution} of $F$ is a homomorphism $\tau : F \rightarrow F$ such that $\tau^2$ is the identity map. For an element $x\in F$, its involution is usually denoted by $\bar{x}$ instead of $\tau(x)$.
\end{df}

% A key example of involution is the complex conjugate. For the sake of convinience, we will assume the involution is the identity map. 

\begin{df}
(Orthogonality) Let $F$ be a tract endowed with an involution $x \mapsto \bar{x}$, and let $E =\{ 1,...,m\}.$ The {\bf inner product} of $X=(x_1,..., x_m) \in \NN[F^\times]^m$ and $Y =(y_1,...,y_m) \in \NN[F^\times]^m$ is defined to be
\[
X\cdot Y := x_1\bar{y}_1+\cdots +x_m\bar{y}_m.
\]

We say that $X$ is {\bf orthogonal} to $Y$ if $X\cdot Y\in N_F$. 
\end{df}

Note that our definition of orthogonality generalizes \cite[Definition 3.4]{Baker-Bowler19}, since for us $X$ and $Y$ are in $\NN[F^\times]^m$ instead of $F^m$.

\medskip

If $S \subseteq \NN[F^\times]^m$, we denote by $S^\perp$ the set of all $X \in \NN[F^\times]^m$ such that $X \perp Y$ for all $Y \in S$. 

\medskip

Let $F$ be a tract endowed with an involution $x \mapsto \bar{x}$, and let $\underline{M}$ be a (classical) matroid with ground set E. 
The following two definitions are taken directly from \cite{Baker-Bowler19}.

\begin{df}
($F$-signature) 
A subset $C$ of $F^E$ is an {\bf $F$-signature of $\underline{M}$} if $C$ satisfies the following properties: \par
(C0) $0\notin C.$ \par 
(C1) If $X\in C$ and $\alpha \in F^\times$ then $\alpha\cdot X \in C$. \par
(C2) Taking supports gives a bijection from the projectivization of $C$ to the set of circuits of $\underline{M}$.
\end{df}

\begin{df}
(Dual pair of $F$-signatures)
Let $C,D \subseteq F^E$. We call $(C,D)$ a {\bf dual pair of $F$-signatures of $\underline{M}$} if: 
\begin{itemize}
\item [(DP1)] $C$ is an F-signature of $\underline{M}$.\par
\item [(DP2)] $D$ is an F-signature of the dual matroid $\underline{M}^*$.\par
\item [(DP3)] $C \perp D$, meaning that $X \perp Y$ for all $X \in C$ and $Y \in D$.
\end{itemize}

\end{df}

For the purposes of this paper, we define a (strong) {\bf $F$-matroid} $M$\footnote{All $F$-matroids in this paper will be strong, so we sometimes omit the modifier.} to be a matroid $\underline{M}$ (called the {\bf underlying matroid} of $M$), together with a dual pair of $F$-signatures of $\underline{M}$. 
The equivalence of this definition with the one given in \cite{Baker-Bowler19} is proved in \cite[Theorem 3.26]{Baker-Bowler19}.
% By \cite[Theorem 3.26]{Baker-Bowler19}, a strong $F$-matroid with underlying matroid $\underline{M}$ is the same thing as a dual pair of $F$-signatures of $\underline{M}$. 
% For our purposes, we will take this as the {\bf definition} of a (strong) $F$-matroid $M$. (All $F$-matroids in this paper will be strong.)

\medskip

We call $C$ (resp. $D$) the set of $F$-circuits (resp. $F$-cocircuits) of $M$, and denote these by $C(M)$ and $C^*(M)$, respectively. 

\begin{df}
We say that $X \in \NN[F^\times]^m$ is a {\bf generalized vector} (resp. {\bf generalized covector}) of $M$ if $X\perp Y$ for every $Y \in C^*(M)$ (resp. for every $Y \in C(M)$). We denote the set of all generalized vectors (resp. covectors) by $\mathcal{V}(M)$ (resp. $\mathcal V^*(M)$).
\end{df}

Note that a {\bf vector} of $M$, in the sense of \cite{Baker-Bowler19}, is just a generalized vector belonging to $F^m$ rather than $\NN[F^\times]^m$, and similarly for covectors. We denote by $V(M)$ (resp. $V^*(M)$) the set of vectors (resp. covectors) of $M$.

\begin{df}
An $F$-matroid $M$ is {\bf strongly perfect} if $\mathcal{V}(M)\perp \mathcal{V^*}(M)$.
A tract $F$ is strongly perfect if every $F$-matroid is strongly perfect.
\end{df}

A strongly perfect tract is obviously perfect. We now show that the converse holds as well:

\begin{thm}
\label{thm:perfectequalsstronglyperfect}
A tract $F$ is perfect if and only if it is strongly perfect.
\end{thm}

For the proof of Theorem~\ref{thm:perfectequalsstronglyperfect}, we will need the following straightforward lemma, whose proof we omit.

\begin{lemma} \label{lem:seriesparallel}
Let $F$ be a tract, and let $M$ be an $F$-matroid on $E$ with underlying matroid $\underline{M}$. 
Let $e \in E$, let $\{ e_1,e_2 \}$ be a 2-element set disjoint from $E$, and let $E' = E \backslash \{ e \} \cup \{ e_1,e_2 \}$.
\begin{enumerate}
    \item There is an $F$-matroid $\sigma_e(M)$ on $E'$, whose underlying matroid is obtained by replacing $e$ with two elements $e_1,e_2$ in series, and whose $F$-circuits $C'$ are given by $C'(f) = C(f)$ for $f \in E$ and $C'(f) = C(e)$ for $f \in \{e_1,e_2 \}$, where $C$ is an $F$-circuit of $M$. The $F$-cocircuits $D'$ of $\sigma_e(M)$ are given by either (i) $D'(f) = D(f)$ for $f \in E$, $D'(e_1)=0$, and $D'(e_2)=D(e)$ or (ii) $D'(f) = D(f)$ for $f \in E$, $D'(e_1)=D(e)$, and $D'(e_2)=0$, for $D$ an $F$-cocircuit of $M$, or (iii) $D'(f) = 0$ for $f \in E$, $D'(e_1)= a \in F^\times$, and $D'(e_2)=-a$.
    \item There is an $F$-matroid $\pi_e(M)$ on $E'$, whose underlying matroid is obtained by replacing $e$ with two elements $e_1,e_2$ in parallel, and whose $F$-cocircuits $D'$ are given by $D'(f) = D(f)$ for $f \in E$ and $D'(f) = D(e)$ for $f \in \{e_1,e_2 \}$, where $D$ is an $F$-cocircuit of $M$. The $F$-circuits $C'$ of $\pi_e(M)$ are given by either (i) $C'(f) = C(f)$ for $f \in E$, $C'(e_1)=0$, and $C'(e_2)=C(e)$ or (ii) $C'(f) = C(f)$ for $f \in E$, $C'(e_1)=C(e)$, and $C'(e_2)=0$, for $C$ an $F$-circuit of $M$, or (iii) $C'(f) = 0$ for $f \in E$, $C'(e_1)= a \in F^\times$, and $C'(e_2)=-a$.
    \end{enumerate}
\end{lemma}

\begin{proof}[Proof of Theorem~\ref{thm:perfectequalsstronglyperfect}]
Let $F$ be a perfect tract, let $M$ be an $F$-matroid, and let $X,Y \in \NN[F^\times]^E$ be elements of $\mathcal{V}(M)$ and $\mathcal{V^*}(M)$, respectively.
We need to show that $\mathcal{V}(M)\perp \mathcal{V^*}(M)$.

To see this, for each $e \in E$ let $k(e) = \min(1, \| X(e) \|)$ and let $\ell(e) = \min(1,\| Y(e) \|)$. 
Let $M'$ be the $F$-matroid on $E'$ obtained from $M$ by replacing each $e \in E$ with $k(e)$ series copies of a bundle of $\ell(e)$ parallel elements. Formally, $M'$ is obtained from $M$ as follows: for each $e \in E$, apply the operator $\sigma_e$ $k(e)-1$ times, thereby replacing $e$ with $k=k(e)$ elements $e_1,\ldots,e_k$; now, for each $i=1,\ldots,k$ apply the operator $\pi_{e_i}$ $\ell(e)-1$ times.

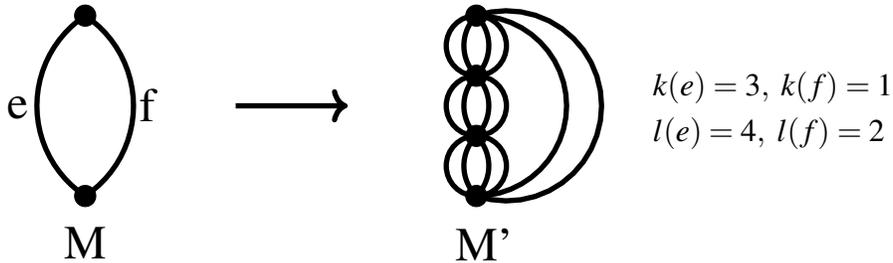
\begin{tikzpicture}[scale = 0.2]
\clip(-9.44773350937211,-10.5) rectangle (55,8);
\draw [shift={(4,0)},line width=2pt]  plot[domain=2.1587989303424644:4.124386376837122,variable=\t]({1*7.211102550927979*cos(\t r)+0*7.211102550927979*sin(\t r)},{0*7.211102550927979*cos(\t r)+1*7.211102550927979*sin(\t r)});
\draw [shift={(-4,0)},line width=2pt]  plot[domain=-0.9827937232473287:0.982793723247329,variable=\t]({1*7.211102550927979*cos(\t r)+0*7.211102550927979*sin(\t r)},{0*7.211102550927979*cos(\t r)+1*7.211102550927979*sin(\t r)});
\draw [->,line width=2pt] (10,0) -- (17.453730053755383,0);
\draw [shift={(28,4)},line width=2pt]  plot[domain=2.356194490192345:3.9269908169872414,variable=\t]({1*2.8284271247461903*cos(\t r)+0*2.8284271247461903*sin(\t r)},{0*2.8284271247461903*cos(\t r)+1*2.8284271247461903*sin(\t r)});
\draw [shift={(24,4)},line width=2pt]  plot[domain=-0.7853981633974483:0.7853981633974483,variable=\t]({1*2.8284271247461903*cos(\t r)+0*2.8284271247461903*sin(\t r)},{0*2.8284271247461903*cos(\t r)+1*2.8284271247461903*sin(\t r)});
\draw [line width=2pt] (26,4) circle (2cm);
\draw [line width=2pt] (26,0) circle (2cm);
\draw [line width=2pt] (26,-4) circle (2cm);
\draw [shift={(28,0)},line width=2pt]  plot[domain=2.356194490192345:3.9269908169872414,variable=\t]({1*2.8284271247461903*cos(\t r)+0*2.8284271247461903*sin(\t r)},{0*2.8284271247461903*cos(\t r)+1*2.8284271247461903*sin(\t r)});
\draw [shift={(24,0)},line width=2pt]  plot[domain=-0.7853981633974483:0.7853981633974483,variable=\t]({1*2.8284271247461903*cos(\t r)+0*2.8284271247461903*sin(\t r)},{0*2.8284271247461903*cos(\t r)+1*2.8284271247461903*sin(\t r)});
\draw [shift={(28,-4)},line width=2pt]  plot[domain=2.356194490192345:3.9269908169872414,variable=\t]({1*2.8284271247461903*cos(\t r)+0*2.8284271247461903*sin(\t r)},{0*2.8284271247461903*cos(\t r)+1*2.8284271247461903*sin(\t r)});
\draw [shift={(24,-4)},line width=2pt]  plot[domain=-0.7853981633974483:0.7853981633974483,variable=\t]({1*2.8284271247461903*cos(\t r)+0*2.8284271247461903*sin(\t r)},{0*2.8284271247461903*cos(\t r)+1*2.8284271247461903*sin(\t r)});
\draw [shift={(26,0)},line width=2pt]  plot[domain=-1.5707963267948966:1.5707963267948966,variable=\t]({1*6*cos(\t r)+0*6*sin(\t r)},{0*6*cos(\t r)+1*6*sin(\t r)});
\draw [shift={(28,0)},line width=2pt]  plot[domain=-1.8925468811915387:1.892546881191539,variable=\t]({1*6.324555320336759*cos(\t r)+0*6.324555320336759*sin(\t r)},{0*6.324555320336759*cos(\t r)+1*6.324555320336759*sin(\t r)});
\draw (-6.3,1.7) node[anchor=north west, scale= 1.5] {e};
\draw (2.5,2.2) node[anchor=north west,scale =1.5] {f};
\draw (-2.5,-7) node[anchor=north west,scale = 1.5] {M};
\draw (23.5,-7) node[anchor=north west,scale = 1.5] {M'};

\draw[color=black] (0.3386793150878914,6.941653749562329) ;
\draw [fill=black] (0,-6) circle (20pt);
\draw [fill=black] (0,6) circle (20pt);
\draw [fill=black] (26,6) circle (20pt);
\draw [fill=black] (26,2) circle (20pt);
\draw [fill=black] (26,-2) circle (20pt);
\draw [fill=black] (26,-6) circle (20pt);
\draw (37,3) node[anchor=north west] { $k(e)=3,\ k(f)=1$};
\draw (37,0) node[anchor=north west] { $l(e)=4,\ l(f)=2$};
\end{tikzpicture}

For each $e \in E$, write $X(e) = a_1(e) + \cdots + a_{\ell(e)}(e)$, with $a_i(e) \in F$, and similarly write 
$Y(e) = b_1(e) + \cdots + b_{k(e)}(e)$.
Define $X' \in \NN[F^\times]^{E'}$ by setting $X'(f) = a_i(e)$ if $f$ is the $i^{\rm th}$ parallel element in any one of the $k(e)$ bundles in series for $i = 1,\ldots,\ell(e)$.
Similarly, define $Y' \in \NN[F^\times]^{E'}$ by setting $Y'(f) = b_j(e)$ if $f$ is any one of the $\ell(e)$ parallel elements in the $j^{\rm th}$ series copy of the bundle of parallel elements replacing $e$ for $j = 1,\ldots,k(e)$.

Using Lemma~\ref{lem:seriesparallel} (which by induction provides us with an explicit description of $C^*(M')$ and $C(M')$, respectively), it is straightforward to check that $X' \in V(M)$ and $Y' \in V^*(M)$. Moreover, we have
$X' \cdot Y' = X \cdot Y$, and since $F$ is perfect, $X \cdot Y = 0$.

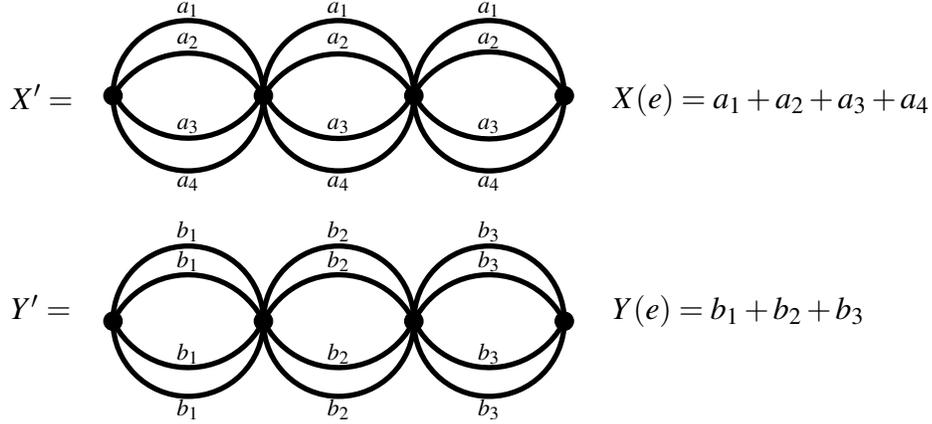
\begin{tikzpicture}[line cap=round,line join=round,>=triangle 45,x=1cm,y=1cm]
\clip(-8,-3.3) rectangle (8,3.7);
\draw (1,1.3) node[anchor=north west] {$X(e)=a_1+a_2+a_3+a_4$};
\draw (1,-1.5) node[anchor=north west] {$Y(e) = b_1+b_2+b_3$};
\draw (-7,1.3) node[anchor=north west] {$X'=$};
\draw (-7,-1.5) node[anchor=north west] {$Y'=$};
\draw [line width=2pt] (-4.5,1) circle (1cm);
\draw [line width=2pt] (-2.5,1) circle (1cm);
\draw [line width=2pt] (-0.5,1) circle (1cm);
\draw [shift={(-4.5,0.4)},line width=2pt]  plot[domain=0.5404195002705842:2.601173153319209,variable=\t]({1*1.16619037896906*cos(\t r)+0*1.16619037896906*sin(\t r)},{0*1.16619037896906*cos(\t r)+1*1.16619037896906*sin(\t r)});
\draw [shift={(-2.5,0.38)},line width=2pt]  plot[domain=0.5549957273385867:2.5865969262512065,variable=\t]({1*1.1766052864066183*cos(\t r)+0*1.1766052864066183*sin(\t r)},{0*1.1766052864066183*cos(\t r)+1*1.1766052864066183*sin(\t r)});
\draw [shift={(-0.5,0.4)},line width=2pt]  plot[domain=0.5317240672588056:2.5922181688182424,variable=\t]({1*1.1833849753989611*cos(\t r)+0*1.1833849753989611*sin(\t r)},{0*1.1833849753989611*cos(\t r)+1*1.1833849753989611*sin(\t r)});
\draw [shift={(-4.5,1.566017083718981)},line width=2pt]  plot[domain=3.665400902449214:5.776616729514695,variable=\t]({1*1.1316236184755692*cos(\t r)+0*1.1316236184755692*sin(\t r)},{0*1.1316236184755692*cos(\t r)+1*1.1316236184755692*sin(\t r)});
\draw [shift={(-2.5,1.5564242562822441)},line width=2pt]  plot[domain=3.6582639833466493:5.784056037480132,variable=\t]({1*1.126390641886703*cos(\t r)+0*1.126390641886703*sin(\t r)},{0*1.126390641886703*cos(\t r)+1*1.126390641886703*sin(\t r)});
\draw [shift={(-0.5,1.583202902165266)},line width=2pt]  plot[domain=3.669240119666738:5.75487814738284,variable=\t]({1*1.1582927483497618*cos(\t r)+0*1.1582927483497618*sin(\t r)},{0*1.1582927483497618*cos(\t r)+1*1.1582927483497618*sin(\t r)});
\draw [line width=2pt] (-4.5,-2) circle (1cm);
\draw [line width=2pt] (-2.5,-2) circle (1cm);
\draw [line width=2pt] (-0.5,-2) circle (1cm);
\draw [shift={(-4.5,-1.5)},line width=2pt]  plot[domain=3.6052402625905993:5.81953769817878,variable=\t]({1*1.118033988749895*cos(\t r)+0*1.118033988749895*sin(\t r)},{0*1.118033988749895*cos(\t r)+1*1.118033988749895*sin(\t r)});
\draw [shift={(-4.5,-2.5)},line width=2pt]  plot[domain=0.4636476090008061:2.677945044588987,variable=\t]({1*1.118033988749895*cos(\t r)+0*1.118033988749895*sin(\t r)},{0*1.118033988749895*cos(\t r)+1*1.118033988749895*sin(\t r)});
\draw [shift={(-2.5,-2.5)},line width=2pt]  plot[domain=0.4636476090008061:2.677945044588987,variable=\t]({1*1.118033988749895*cos(\t r)+0*1.118033988749895*sin(\t r)},{0*1.118033988749895*cos(\t r)+1*1.118033988749895*sin(\t r)});
\draw [shift={(-2.5,-1.5)},line width=2pt]  plot[domain=3.6052402625905993:5.81953769817878,variable=\t]({1*1.118033988749895*cos(\t r)+0*1.118033988749895*sin(\t r)},{0*1.118033988749895*cos(\t r)+1*1.118033988749895*sin(\t r)});
\draw [shift={(-0.5,-1.5)},line width=2pt]  plot[domain=3.6052402625905993:5.81953769817878,variable=\t]({1*1.118033988749895*cos(\t r)+0*1.118033988749895*sin(\t r)},{0*1.118033988749895*cos(\t r)+1*1.118033988749895*sin(\t r)});
\draw [shift={(-0.5,-2.5)},line width=2pt]  plot[domain=0.4636476090008061:2.6409055013682043,variable=\t]({1*1.118033988749895*cos(\t r)+0*1.118033988749895*sin(\t r)},{0*1.118033988749895*cos(\t r)+1*1.118033988749895*sin(\t r)});
\begin{scriptsize}
\draw [fill=black] (-5.5,1) circle (3.5pt);
\draw [fill=black] (-3.5,1) circle (3.5pt);
\draw [fill=black] (-1.5,1) circle (3.5pt);
\draw [fill=black] (0.5,1) circle (3.5pt);
\draw [fill=black] (-5.5,-2) circle (3.5pt);
\draw [fill=black] (-1.5,-2) circle (3.5pt);
\draw [fill=black] (0.5,-2) circle (3.5pt);
\draw [fill=black] (-3.5,-2) circle (3.5pt);
\draw[color=black] (-4.5,2.15) node {$a_1$};
\draw[color=black] (-4.5,1.73) node {$a_2$};
\draw[color=black] (-4.5,0.60) node {$a_3$};
\draw[color=black] (-4.5,-0.17) node {$a_4$};
\draw[color=black] (-2.5,2.15) node {$a_1$};
\draw[color=black] (-2.5,1.73) node {$a_2$};
\draw[color=black] (-2.5,0.60) node {$a_3$};
\draw[color=black] (-2.5,-0.17) node {$a_4$};
\draw[color=black] (-0.5,2.15) node {$a_1$};
\draw[color=black] (-0.5,1.73) node {$a_2$};
\draw[color=black] (-0.5,0.60) node {$a_3$};
\draw[color=black] (-0.5,-0.17) node {$a_4$};
\draw[color=black] (-4.5,-0.8) node {$b_1$};
\draw[color=black] (-4.5,-1.2) node {$b_1$};
\draw[color=black] (-4.5,-2.43) node {$b_1$};
\draw[color=black] (-4.5,-3.17) node {$b_1$};
\draw[color=black] (-2.5,-0.8) node {$b_2$};
\draw[color=black] (-2.5,-1.2) node {$b_2$};
\draw[color=black] (-2.5,-2.43) node {$b_2$};
\draw[color=black] (-2.5,-3.17) node {$b_2$};
\draw[color=black] (-0.5,-0.8) node {$b_3$};
\draw[color=black] (-0.5,-1.2) node {$b_3$};
\draw[color=black] (-0.5,-2.43) node {$b_3$};
\draw[color=black] (-0.5,-3.17) node {$b_3$};
\end{scriptsize}
\end{tikzpicture}

\end{proof}

The following propositions concern the behavior of generalized vectors and covectors with respect to deletion and contraction. For (non-generalized) vectors and covectors, the corresponding results are proved as Propositions 4.3 and 4.4, respectively, in Laura Anderson's paper \cite{Anderson19}. The proofs given in \cite{Anderson19} work {\bf mutatis mutandis} for generalized vectors; alternatively, one can reduce the generalized case to the one treated in \cite{Anderson19} using a trick similar to the one in the proof of Theorem~\ref{thm:perfectequalsstronglyperfect}.

\begin{prop} \label{minor1}
$ \{  Y \setminus e \ | \ Y\in \mathcal{V}^*(M), \ Y(e)=0 \} \subseteq \mathcal{V}^*(M/e)$ and
$ \{  X \setminus e \ | \ X\in \mathcal{V}(M), \ X(e)=0 \} \subseteq \mathcal{V}(M\setminus e)$.
\end{prop}

\begin{prop} \label{minor2}
$ \{  Y \setminus e \ | \ Y\in \mathcal{V}^*(M) \} \subseteq \mathcal{V}^*(M\setminus e)$ and
$ \{  X \setminus e \ | \ X\in \mathcal{V}(M) \} \subseteq \mathcal{V}(M/e)$.
\end{prop} 

In particular, the contraction of a generalized vector is again a generalized vector, and the deletion of a generalized covector is again a generalized covector.

\section{Strongly fused tracts are strongly perfect}
\label{sec:SFTSP}

Recall the modified strong fusion axiom:

\begin{itemize}
\item[(MSF)] If $\alpha + \gamma$ and $\beta - \gamma$ are in $N_F$ with $\alpha, \beta, \gamma \in \NN[F^\times]$ and either $\gamma = 0$ or $\gamma \not\in N_F$, and if $\| \alpha + \beta \| \geq 4$, then $\alpha + \beta \in N_F$.
\end{itemize}  

Our goal in this section is to prove the following theorem, which generalizes Theorem~\ref{thm:mainthm}:

\begin{thm}
\label{thm:strongmainthm}
If a tract $F$ satisfies (MSF) then $F$ is strongly perfect.
\end{thm}

The following is an example of a tract that satisfies (MSF) but not (SF).
\begin{ex} \label{ex:MSFbutnotSF}
Let $\P$ be the phase hyperfield and take the tract embedding $(\P^\times, N_\P)$. Letting $N_{\P'} = N_\P \cup \N[\P^\times]_{\geq 4}$, it is straightforward to show that $\P' = (\P^\times, N_{\P'})$ satisfies the tract axiom and axiom (MSF). However, it does not satisfy the strong fusion axiom (SF). Let $\alpha = 1$, $\beta = 1+1$, and $\gamma = (-1)+(-1)$. Then, we have the following,
\begin{align*}
   \alpha + \gamma = 1+(-1) + (-1) &\in N_{\P'}. \\        
    \beta - \gamma = 1+1-((-1)+(-1))  &\in \N[\P^\times]_{\geq 4} \subseteq N_{\P'}. \\
     \gamma = (-1) + (-1) &\notin N_{\P'}. \\
   \alpha + \beta = 1 + 1 + 1 &\notin N_{\P'}.
\end{align*}
\end{ex}

\begin{rem}
Generalizing (part of) Example~\ref{ex:MSFbutnotSF}, it is straightforward to check that weak hyperfields in the sense of \cite[Example 2.14]{Baker-Bowler19} satisfy (MSF).
\end{rem}

The proof of Theorem~\ref{thm:strongmainthm} is fairly long and technical, so it will be broken up into a number of smaller and hopefully more digestible pieces.

\begin{lemma} \label{lem:ideal}
If $F$ is a tract satisfying the idyll property (I), then for $\alpha \in \NN[F^\times]$ and $\beta \in N_F$ we have $\alpha\beta \in N_F$. 
\end{lemma}

\begin{proof}
If $\alpha = 0$, this is obvious. Otherwise, write $\alpha = \sum_{i=1}^k x_i$ with $x_i \in F^\times$ and inductively apply (I).
\end{proof}

\begin{lemma}
\label{MSF_implies_F}
If a tract $F$ satisfies the modified strong fusion axiom (MSF), then it also satisfies the fusion axiom (F).
\end{lemma}
\begin{proof}
It suffices to show that (F) is satisfied when $||\alpha + \beta ||\leq 3$.\par

If $z = 0$, then $\alpha$ and $\beta$ are either zero or they belong to $N[F^\times]_{\geq 2}$. Hence, either at least one of $\alpha, \ \beta$ is 0 or $||\alpha + \beta|| \geq 4$. \par 

Assume $z\in F^\times$. If $||\alpha||=1$ or $||\beta|| = 1$, the result is clear. 
% without loss of generality, assume $||\beta|| = 1$. Then, $\beta = z$ and $\alpha + \beta = \alpha + z \in N_F$. 
Otherwise, both $||\alpha||$ and $||\beta||$ are at least 2.
\end{proof}

\begin{prop} \label{MSF'}
Let $F$ be a tract which satisfies (MSF). Suppose $\alpha,\beta,\gamma,\delta \in \NN[F^\times]$ with $\gamma\notin N_F$, $\alpha+\beta\gamma \in N_F$, $\delta- \gamma \in N_F$, and $||\alpha+\beta \delta||\geq 4$. Then $\alpha+\beta \delta \in N_F.$ 
\end{prop}

\begin{proof}
We proceed by induction on $||\beta||$. For $||\beta||=1$, the result follows immediately from (MSF). \par 
Assuming the result holds for $||\beta||<k$ with $k \geq 2$, we will prove it for $\beta = y_1+\cdots+ y_k \in \NN[F^\times]$. \par 

If $||\gamma|| = 1$, then since $\alpha+\beta\gamma = \alpha+y_1 \gamma + \cdots + y_{k-1}\gamma + y_k \gamma\in N_F$ and 
$y_k \delta - y_k \gamma \in N_F$, it follows from the fusion axiom (F)
%\footnote{\mb I added a hypothesis to (MSF) that $F$ should satisfy the fusion axiom as well.}%
that $\alpha + y_1 \gamma + \cdots + y_{k-1}\gamma + y_k \delta \in N_F.$ Letting $\alpha' = \alpha + y_k \delta$ and $\beta' = y_1+\cdots+ y_{k-1}$, it follows from the induction hypothesis that $\alpha+\beta \delta = \alpha'+\beta' \delta \in N_F.$ \par 

If $||\gamma|| \geq 2$, then since $\| \beta \|=k \geq 2$ as well we have $||\alpha+\beta \gamma||\geq4$ and $||\alpha + y_1 \gamma + \cdots + y_{k-1}\gamma + y_k \delta|| \geq 4$. 
% is a number between $||\alpha+\beta \gamma||$ and $||\alpha+\beta \delta||$, so it must be greater than or equal to 4. 
By (MSF) we have $\alpha + y_1 \gamma + \cdots + y_{k-1}\gamma + y_k \delta \in N_F$. As in the previous case, this implies by induction that $\alpha+\beta \delta \in N_F.$
\end{proof}

Next, we introduce the ``wedge product'' $X\wedge_e Y$ for $X,Y\in \NN[F^\times]^E$ and $e \in E$ 
% (compare with \cite{Dress-Wenzel92}) 
which will allow us to perform an analogue of (co)circuit elimination for generalized (co)vectors of matroids over tracts satisfying (MSF).
% functions as a sort of ``vector elimination'' for matroids over pastures that we can do to the vectors over pastures. 

\begin{df}
For $X,Y\in \NN[F^\times]^E$, and $e\in E$ we define $X\wedge_e Y \in \NN[F^\times]^E$ by 
\begin{equation*}
    X\wedge_e Y(f) = \begin{cases}
               0                          &\text{if } f = e,\\            
               Y(e)X(f)- X(e)Y(f) & \text{if } f \neq e.
                     \end{cases}
\end{equation*}
\end{df}

The following result and its proof were inspired by \cite[Lemma 2.4]{Dress-Wenzel92} (which is proved in \cite[Lemma 3.2]{Dress86}).

\begin{prop} \label{wedge}
Let F be a tract which satisfies (MSF), and let $M$ be an $F$-matroid on $E=\{1,2,...,m\}$. For any $X,Y\in \mathcal{V^*}(M)$ and $e \in E$ we have $X\wedge_e Y\in \mathcal{V^*}(M)$.
\end{prop}

\begin{proof}
% To prove that $X\wedge_e Y\in \mathcal{V^*}(M)$, 
It suffices to show that for any $C\in C(M)$, $C\cdot (X\wedge_e Y) \in N_F$. 
Note for later reference that 
\[
C\cdot (X\wedge_e Y) = \sum_{f \in \underline{C}\setminus e}\left( C(f)X(e)Y(f)-C(f)Y(e)X(f) \right).
\]

We will consider the following two cases.

\medskip

{\bf Case 1:} $||C\cdot (X\wedge_e Y)||\geq 4$.

\medskip

We have the following subcases:

\begin{enumerate}
\item[•] If $C(e)=0$, then $C\cdot (X\wedge_e Y) = Y(e) \left( C \cdot X \right) - X(e) \left( C \cdot Y \right) \in N_F$ by the idyll axiom (I).\par
% then it is obvious that $C\perp (X\wedge_e Y)$.\par
\item[•] If $X(e),Y(e)\in N_F$, then $X\wedge_e Y(f) \in N_F$ for all $f \in E$ and thus $C\cdot (X\wedge_e Y) \in N_F$ by (I). \par
% every term of $X\wedge_e Y$ is contained in $N_F$, it is perpendicular to everything by the fusion axiom. \par
\item[•] Suppose that either $X(e) \notin N_F$ or $Y(e)\notin N_F$. By symmetry, we can assume $Y(e) \notin N_F$.
Since $X\perp C$, we have $C(e)X(e)+ \underset{f \in E\setminus \{e\}}{\Sigma} C(f)X(f) \in N_F$ and thus 
\begin{equation} \label{eq:1}
C(e)X(e)Y(e)+\underset{f \in E\setminus \{e\}}{\Sigma} C(f)X(f)Y(e) \in N_F. 
\end{equation}

And since $Y\perp C$,
\begin{equation} \label{eq:2}
C(e)Y(e)+\underset{f \in E\setminus \{e\}}{\Sigma} C(f)Y(f)\in N_F.
\end{equation}

Since $C(e)Y(e)\notin N_F$, we can apply Proposition~\ref{MSF'} to \eqref{eq:1} and \eqref{eq:2}
with $\alpha = \underset{f \in E\setminus \{e\}}{\Sigma} C(f)X(f)Y(e)$, $\beta = X(e)$, $\gamma = C(e)Y(e)$, and
$\delta = -\underset{f \in E\setminus \{e\}}{\Sigma} C(f)Y(f)$.
As a result, we get $C \perp (X\wedge_e Y)$ .
\end{enumerate}

\medskip

{\bf Case 2:} $||C\cdot (X\wedge_e Y)|| < 4$.

\medskip

%Deleting all elements of $E$ which are not in the support of $C$, we consider the matroid $M' = M\setminus (E\setminus \underline{C})$. By \cite[Theorem 3.28]{Baker-Bowler19}, $C|_{M'}$ is still a circuit of $M'$.\footnote{\mb Why do we need to consider $M'$? Can't we just delete the two sentences which make reference to $M'$?}\par %
It is easy to check that if any one of $\{C(e), X(e), Y(e)\}$ is 0, then $(X\wedge_e Y)\perp C$. Assuming none of them is 0, we have the following subcases:

\begin{enumerate}
\item[•] If $||X(e)||\geq 2$ and $||Y(e)||\geq 2$, then either $\underline{X} \cap \underline{C} = \{e\}$ or $\underline{Y} \cap \underline{C}=\{e\}$, since otherwise $||C\cdot (X\wedge_e Y)||\geq  4$. Without loss of generality, assume $\underline{X} \cap \underline{C} = \{e\}$. Then, since $X\perp C$,  $X(e)\in N_F$ and therefore
\[
C\cdot (X\wedge_e Y) = \sum_{f \in E\setminus \{ e \}} C(f)X(e)Y(f) \in N_F.
\]
% it is easy to check $C\cdot (X\wedge_e Y) \in N_F$ in this case. 
\item[•] If $||X(e)||=1$ or $||Y(e)||=1$, then either $|\underline{X} \cap \underline{C}|=\{e,f\}$ with $||X(f)||=1$ or $|\underline{Y} \cap \underline{C}|=\{e,f\}$ with $||Y(f)||=1$, since otherwise $||C\cdot (X\wedge_e Y)||\geq  4$.\par
% \footnote{\mb I combined what were previously two separate cases into one case here. Please double-check that I did this correctly.\tz I think that is correct.} \par
% It follows that $|\underline{X} \cap \underline{C}|\leq 2$, since otherwise $||C\cdot (X\wedge_e Y)||\geq  4$ as $C(e) \neq 0$ and $C \cdot Y = 0$ imply $\underline{Y} \neq \{ e \}$. \par 

% If $\underline{X} \cap \underline{C} = \{e\}$, then $||X(e)|| = 1$ implies that $\underline{X} = \{ e \}$, contradicting the fact that $C(e) \neq 0$ and $C \cdot X = 0$.  \par

By symmetry, we may assume without loss of generality that $|\underline{X} \cap \underline{C}|=\{e,f\}$ and $||X(f)||=1$.
% If $\underline{X} \cap \underline{C} = \{e,f\}$, then $||X(f)||=1$, since otherwise $||C\cdot (X\wedge_e Y)||\geq  4$. 
Since $X\perp C$, $X(e)C(e)+X(f)C(f) \in N_F$. And since both $X(e)C(e)$ and $X(f)C(f)$ belong to $F^\times$, we must have $X(e)C(e)=-X(f)C(f)$. Therefore
\begin{align}
    C\cdot (X\wedge_e Y) &= \sum_{g\in \underline{C}\setminus e}\left( C(g)X(e)Y(g)-C(g)Y(e)X(g) \right) \\        
          &= \left(\sum_{g\in \underline{C}\setminus e} C(g)X(e)Y(g)\right) -C(f)Y(e)X(f)     \\
          &= \left(\sum_{g\in \underline{C}\setminus e}C(g)X(e)Y(g)\right) +C(e)Y(e)X(e)     \\
          &= \sum_{g\in \underline{C}} C(g)X(e)Y(g)      \\
          &= X(e)\left( C \cdot Y \right) \in N_F.      
\end{align}
% \item[•] If $||X(e)||=1$ and $||Y(e)||=1$, then either $|\underline{X} \cap \underline{C}|=\{e,f\}$ with $||X(f)||=1$ or $|\underline{Y} \cap \underline{C}|=\{e,f\}$ with $||Y(f)||=1$, since otherwise $||C\cdot (X\wedge_e Y)||\geq  4$. In either situation we have $C\cdot (X\wedge_e Y) \in N_F$ for the same reason as in the previous case. 
\end{enumerate}
\end{proof}

The following result and its proof were inspired by \cite[Lemma 2.6]{Dress-Wenzel92}.

\begin{prop} \label{sum'}
Suppose $F$ is a tract satisfying (MSF). Let $X_1,...,X_n\in \NN[F^\times]$ be such that $X_I := \sum_{i\in I} X_i\in N_F$ 
for every $I\subset \{1,2,...,n\}$ with $n-2\leq|I|\leq n-1$. Then $\sum_{i=1}^n X_i \in N_F$.
% \footnote{We should point out the very similar Proposition in Dress--Wenzel's paper.\tz I had a very short comparison(of prop 4.6 and 4.5) at the end of the paper shall we move them over here?}
\end{prop}

\begin{proof}
We may assume, without loss of generality, that
$X_i\notin N_F$ for all $i$ and $X_i+X_j \notin N_F$ for all $i \neq j$, since otherwise 
\[
X_i+\sum_{\substack{j=1 \\ j \neq i}}^n X_j \in N_F
% \overset{n}{\underset{j=1,j\neq i}{\sum}}X_j \in N_F
\]
or 
\[
X_i+X_j+\sum_{\substack{k=1 \\ k \neq i,j}}^n X_k \in N_F.
% X_i+X_j+\overset{n}{\underset{k=1,k\neq i,j}{\sum}}X_k \in N_F.
\]

Let $J$ be a maximum non-empty proper subset of $\{ 1,\ldots, n \}$ such that $X_J := \underset{j\in J}{\sum} X_j\notin N_F$. Since $J$ is proper, $|J|\leq n-3.$ By symmetry, we may assume without loss of generality that $J\subseteq\{4,5,...,n\}$ and that $3 \in I := J^c \backslash \{ 1,2 \}.$
% \footnote{\mb I added this and changed $X_3$ to $X_I$ throughout the proof; otherwise we don't get $\sum_{i=1}^n X_i$ at the end of the proof. OK? \tz Yes, you are right. The proof was used for $J = \{4,5,...,n\}$. I forgot to change it. I think you meant $I:=J^c\setminus \{1,2\}$ instead of $I:=J\setminus \{1,2\}$, am I right? \mb Yes, thanks!}

\medskip

From the maximality of $J$, we have
\[
X_1+X_I+X_J\in N_F, \; X_1+X_2+X_J\in N_F
\]
and since $X_J\notin N_F$ by assumption, (MSF) implies\footnote{Recall that $X_1 - X_1$ is not the same thing as zero in $\NN[F^\times]$!} 
% \tz Yes, I was just worried if it will look strange, but it in fact looks nice.}
% \overset{\text{4-term} (MSF)}{\Longrightarrow} 
\begin{equation} \label{eq:nullcombo1}
X_1+X_I - X_1 - X_2 \in N_F .
\end{equation}

Since $X_1 + X_J \in N_F$ and $X_1 \not\in N_F$, (MSF) applied to \eqref{eq:nullcombo1} yields
\begin{equation} \label{eq:nullcombo2}
X_1 -X_I+X_2+X_J \in N_F .
\end{equation}

Similarly, since $X_1+X_2+X_J \in N_F$ and $X_J\notin N_F$, (MSF) applied to \eqref{eq:nullcombo2} yields
\begin{equation} \label{eq:nullcombo3}
X_1 - X_I + X_2 - X_1 - X_2 \in N_F.
\end{equation}

Finally, since $X_1+X_2+X_J\in N_F$ and $X_1+X_2 \notin N_F$, (MSF) applied to \eqref{eq:nullcombo3} yields
\[
\sum_{i=1}^n X_i= X_1+X_2+X_I+X_J \in N_F.
\]

\end{proof}

Our next goal is to prove that for any generalized vector $X$ and any generalized covector $Y$ such that $X\cdot Y$ has at most three terms, we have $X\perp Y$. 
% This is one of the reasons why we want to consider the (MSF) instead of (SF). 
We first recall the following key lemma from \cite{Baker-Bowler19}: 
\begin{lemma} \label{supp}
Let $X$ be a generalized vector of $M$ and choose $e\in E$ with $X(e)\notin N_F$. Then there is
some circuit $C$ with $e\in \underline{C}\subset \underline{X}$.
\end{lemma}

Although \cite[Lemma 3.43]{Baker-Bowler19} is stated in the language of fuzzy rings, the same (straightforward) proof works for generalized vectors in our sense. 

\begin{prop} \label{lower-term}
Let $F$ be a tract and let $M$ be an $F$-matroid. If $X\in \mathcal{V}(M)$ and $Y\in \mathcal V^*(M)$ satisfy $||X\cdot Y||\leq 3$, then $X\perp Y$.
\end{prop}

\begin{proof}
We will treat each of the four possibilities for $|\underline X\cap \underline Y|$ separately.

\medskip

{\bf Case 0:} $|\underline X\cap \underline Y|=0$.\par

In this case, $X \cdot Y = 0 \in N_F$.

\medskip

{\bf Case 1:} $|\underline X\cap \underline Y|=1$.\par

Writing $\underline X\cap \underline Y=\{e \}$, we claim that either $X(e)$ or $Y(e)$ is null. Indeed, suppose that neither $X(e)$ nor $Y(e)$ belongs to $N_F$. Then by Lemma \ref{supp}, there is a circuit $C$ and cocircuit $D$ such that $e\in \underline{C}\subset \underline{X}$ and $e\in \underline{D}\subset \underline{Y}$. But this implies $\underline{C}\cap \underline{D}=\{e\}$, which is impossible.

Note that a tract that satisfies (F) is an idyll. We therefore have $X \cdot Y = X(e) Y(e) \in N_F$.

\medskip

{\bf Case 2:} $|\underline X\cap \underline Y|=2$. \par

Writing $\underline X\cap \underline Y=\{e,f\}$, we observe that, since $X\cdot Y$ has at most 3 terms, at least three of $X(e),X(f),Y(e),Y(f)$ must lie in $F^\times$ (and not just $\NN[F^\times]$). Without loss of generality, we may suppose that $X(e), X(f), Y(e)\in F^\times$. In particular, these values are non-null. By Lemma \ref{supp}, there exist a a circuit $C$ such that $e\in \underline{C} \subset \underline{X}$ and a cocircuit $D$ such that $e\in \underline{D} \subset \underline{Y}$. Thus $e \in \underline{C}\cap \underline{D} \subset \underline{X} \cap \underline{Y} = \{ e,f \}$ and $|\underline{C}\cap \underline{D}| \neq 1$, from which it follows that $\underline{C}\cap \underline{D} = \{ e,f \}$. 
We therefore have the following relations: 
\[
C(e)D(e)+C(f)D(f)\in N_F \Rightarrow C(f)=-\frac{ \ C(e)D(e)}{D(f)}.
\]
\[
X(e)D(e)+X(f)D(f)\in N_F \Rightarrow X(f)=-\frac{ \ X(e)D(e)}{D(f)}.
\]
% The second equation holds because both $X(e)$ and $X(f)$ are invertible. 

% As $C\perp Y$, implies $X\perp Y$. 
Since
\[
X \cdot Y = X(e)Y(e) + X(f)Y(f) = X(e) \left( Y(e) - \frac{D(e)}{D(f)} Y(f) \right)
\]
and $C \perp Y$ implies that
\[
C \cdot Y = C(e)Y(e) + C(f)Y(f) = C(e) \left( Y(e) - \frac{D(e)}{D(f)} Y(f) \right) \in N_F,
\]
we have $Y(e) - \frac{D(e)}{D(f)} Y(f) \in N_F$ as well and thus $X \cdot Y \in N_F$ by Lemma~\ref{lem:ideal}.

\medskip

{\bf Case 3:} $|\underline X\cap \underline Y|=3$.\par

Suppose $\underline X\cap \underline Y=\{e,f,g\}$. Let $I = E - \underline{X}, J = \underline{X}\setminus\underline{Y}$, and let $M' = M \setminus I / J$ be the corresponding minor on $E' = \{ e,f,g \}$.
% We can take a minor $M'$ on $E'=\{e,f,g\}$ by taking deletion for all $i\notin \underline{X}$ and contraction for all $j\in \underline{X}\setminus \underline{Y}$. 
% By Propositions~\ref{minor1} and \ref{minor2}, respectively, $X'=X|_{E'}$ is a vector of $M'$ and $Y'=Y|_{E'}$ is a covector of $M'$.\par 
By Propositions~\ref{minor1} and \ref{minor2}, the natural restrictions $X'=X|_{E'}$ and $Y'=Y|_{E'}$ are vectors and covectors of $M'$, respectively. 
Since $\underline{X}' = \underline{Y}' = E'$, the matroid $\underline{M}'$ has no loops or coloops (for example, if $e$ is a loop then $C' = \{ e \}$ is a circuit with $C'\cdot Y=Y(e)C(e) \in F^\times$, which is impossible, and similarly for coloops). By \cite[Table 1.1]{Oxley92}, the only matroids on 3 elements with no loops or coloops are the uniform matroids $U_{1,3}$ and $U_{2,3}$. By duality, we may assume without loss of generality that $\underline{M}' = U_{2,3}$. In this case, $M'$ has exactly one $F$-circuit $C$, with $\underline{C}=\{e,f,g\}$, and three $F$-cocircuits $D_e, \ D_f$, and $D_g$ with $\underline {D_e}=\{f,g\}, \ \underline {D_f}=\{e,g\}$ and $\underline {D_g}=\{e,f\}$.\par

The fact that $\underline{X}' = \underline{Y}' = E'$, combined with $\| X \cdot Y \| = 3$, implies that $X'$ is a vector (not a generalized vector) of $M'$ and $Y'$ is a covector.
By \cite[Lemma 4.19]{Baker-Bowler19}, the $F$-circuits of $M'$ are exactly the vectors of $M'$ having minimal non-empty support; together with Axiom (C2) in \cite[Definition 3.11]{Baker-Bowler19}, this implies in our situation that $X$ is a scalar multiple of $C$. It follows that $X \perp Y$.
\end{proof}

We have finally put together all the ingredients needed to prove the main theorem. The structure of the following proof is modeled on \cite[Proof of Theorem 2.7]{Dress-Wenzel92}.

\begin{proof}[Proof of Theorem \ref{thm:strongmainthm}]
Assume there exists a non-perfect $F$-matroid $M$ on some set $E$, and choose $|E|$ as small as possible. Then by definition, there exist $X\in \mathcal{V}(M)$ and $Y\in \mathcal{V}^*(M)$ such that $X \cdot Y \notin N_F$. By Prop \ref{lower-term}, $\| X\cdot Y \| \geq 4$.  \par
\begin{claim}
$\underset{e\in E \setminus \{ e_0 \}}\sum Y(e)C(e)\in N_F$ for all $e_0 \in E$ and all $C\in C(M)$. 
\end{claim}

Assume for the sake of contradiction that there exist $e_0\in E$ and $C\in C(M)$ such that $\underset{e\in E'}{\sum} Y(e)C(e)\notin N_F$, where $E' = E\setminus \{e_0\}$. 
%Since $\underset{e\in E'}{\sum} Y(e)C(e) + Y(e_0)C(e_0) = \underset{e\in E}{\sum} Y(e)C(e)\in N_F$, we have $Y(e_0)C(e_0)\neq 0$. In particular, $C(e_0)$ is not 0 and hence invertible. From the previous claim, $Y(e_0) \notin N_F$, we must have $Y(e_0)C(e_0)\notin N_F$.
%Since $\underset{e\in E}{\sum} Y(e)C(e)\in N_F$ and (MSF) implies (I), we must have $Y(e_0)C(e_0)\notin N_F$, and in particular $C(e_0) \neq 0$. \par

Let $M' = M\setminus e_0$. By Propositions~\ref{minor1} and \ref{minor2},
\begin{center}
$Y|_{E'}\in \mathcal{\mathcal{V}^*}(M')$ \quad and \quad
$(X\wedge_{e_0}C)|_{E'}\in \mathcal{\mathcal{V}}(M')$.
\end{center}
By the minimality of $E$, $M'$ is perfect, and therefore
\[
\sum_{e\in E'} Y(e)(X\wedge_{e_0}C) \in N_F.
\]
Explicitly, this means that
\begin{equation} \label{eq:cancel1}
C(e_0)\sum_{e\in E'} Y(e)X(e)- X(e_0)\sum_{e\in E'} Y(e)C(e) \in N_F.
\end{equation}
We also have
\begin{equation} \label{eq:cancel2}
Y\cdot C = Y(e_0)C(e_0) + \sum_{e\in E'} Y(e)C(e)\in N_F.
\end{equation}

Since $\| X\cdot Y \| \geq 4$, we may apply Proposition~\ref{MSF'} with $\beta = X(e_0)$ and $\gamma = -\sum_{e\in E'} Y(e)C(e) \notin N_F$ to \eqref{eq:cancel1} and \eqref{eq:cancel2}
to obtain
% $\alpha = C(e_0)\sum_{e\in E'} Y(e)X(e)$, $\beta = Y(e_0)C(e_0), \gamma = -\sum_{e\in E'} Y(e)C(e)$, and $\delta = X(e_0)$ to obtain
\[
C(e_0) (X \cdot Y) = C(e_0)\sum_{e\in E'} Y(e)X(e)+Y(e_0)C(e_0)X(e_0)\in N_F, 
\]
which means that $X \perp Y$, a contradiction. \par

\begin{claim}
$X \perp Y$ (a contradiction which finishes the proof). 
% $X$ is orthogonal to $Y$, which contradict the hypothesis.
\end{claim}

From the previous claim, 
$\underset{e\in E \setminus \{ e_0 \}}\sum Y(e)C(e)\in N_F$ for all $e_0 \in E$ and all $C\in C(M)$. 
By symmetry, we also have $\underset{e\in E \setminus \{ e_0 \}}\sum X(e)D(e)\in N_F$ for all $e_0 \in E$ and $D\in C^*(M)$.
Thus $X|_{E\setminus\{e\}}$ is contained in both $\mathcal{V}(M\setminus e)$ and $\mathcal{V}(M/e)$ for all $e \in E$, and similarly $Y|_{E\setminus\{e\}}$ is contained in both $\mathcal{V^*}(M\setminus e)$ and $\mathcal{V^*}(M/ e)$ for all $e\in E$. 
Since both $M\setminus e$ and $M/ e$ are perfect, we have
\begin{equation} \label{eq:allbut1}
\sum_{e\in E\setminus\{e_0\}} X(e)Y(e) \in N_F \quad \forall  e_0\in E.  
\end{equation}
Moreover for every $e_1,e_2 \in E$ with $e_1\neq e_2$, we have 
\begin{center}
$X|_{E\setminus \{e_1\}} \in \mathcal V(M\setminus \{e_1\})$ \qquad and \qquad
$Y|_{E\setminus \{e_2\}} \in \mathcal V^*(M/ \{e_2\})$.
\end{center}
By Propositions \ref{minor1} and \ref{minor2}, 
\begin{center}
$X|_{E\setminus \{e_1,e_2\}} \in \mathcal V(M\setminus \{e_1\}/\{e_2\})$ \qquad and \qquad
$Y|_{E\setminus \{e_1,e_2\}} \in \mathcal V^*(M/ \{e_2\} \setminus \{e_1\})$.
\end{center}

Since $M'' = M \setminus e_1 / e_2 = M / e_2 \setminus e_1$ is perfect, we have
\begin{equation} \label{eq:allbut2}
\sum_{e\in E\setminus\{e_1,e_2\}} X(e)Y(e) \in N_F \quad \forall  e_1,e_2\in E. 
\end{equation}

Applying Proposition~\ref{sum'} to \eqref{eq:allbut1} and \eqref{eq:allbut2} shows that $X\perp Y$ as claimed.
\end{proof}

\section{Comparison with the work of Dress--Wenzel}
\label{sec:comparison}

In this section we briefly compare our results with those in \cite{Dress-Wenzel92}.
% and show that our main theorem implies theirs but not conversely. \mb I removed this for now!

\medskip

For ease of exposition, we work with Lorscheid's ``simplified fuzzy rings''. It is proved in \cite[Appendix B]{Baker-Bowler19} that every fuzzy ring in the sense of Dress--Wenzel is weakly isomorphic to a simplified fuzzy ring, and it is proved in \cite[Theorem 2.21]{Baker-Lorscheid18} that the category of simplified fuzzy rings can be identified with a full subcategory of the category of tracts. In particular, every simplified fuzzy ring can be identified in a natural way with a tract.\footnote{For a discussion of which tracts come from simplified fuzzy rings, see \cite[Example 2.11]{Baker-Lorscheid18}.}

\medskip

A {\bf simplified fuzzy ring} in the sense of Lorscheid is a tuple $(K,+,\cdot,\epsilon,K_0)$ where $(K,+,\cdot)$ is a commutative semiring equal to $\NN[K^\times]$ and such that $\epsilon,K_0$ satisfy the following axioms:

\begin{enumerate}
    \item[(FR4)] $K_0$ is a proper semiring ideal, i.e., $K_0+K_0\subseteq K_0,$ $K \cdot K_0 \subseteq K_0$, $0\in K_0$ and $1 \notin K_0$.
    \item[(FR5)] For $\alpha \in K^*$ we have $1+\alpha \in K_0$ if and only if $\alpha = \epsilon$.
    \item[(FR6)] If $x_1,x_2,y_1,y_2 \in K$ and $x_1+y_1$, $x_2+y_2\in K_0$ then $x_1\cdot x_2 + \epsilon \cdot y_1 \cdot y_2 \in K_0$.
\end{enumerate}

To give a sufficient condition for perfection, Dress and Wenzel introduce the following variant of (FR6):
\begin{enumerate}
    \item [(FR6$''$)] If $\kappa, \lambda_1, \lambda_2 \in K,$ $ \mu \in K \setminus K_0$, and $\kappa + \mu \cdot \lambda_1, \ \mu + \lambda_2\in K_0$ then $\kappa + \epsilon \cdot \lambda_1 \cdot \lambda_2 \in K_0$.
\end{enumerate}

\begin{prop} \label{prop:FR6''impliesSF}
A simplified fuzzy ring satisfying (FR6$''$), when viewed as a tract, satisfies the strong fusion axiom (SF).
\end{prop}
 
\begin{proof}
Given a simplified fuzzy ring $K$, let $F_K$ denote the tract associated to it. If $\alpha + \gamma$ and $\beta - \gamma$ are in $N_{F_K}$, then in terms of the fuzzy ring we have $\alpha + \gamma, \beta - \gamma \in K_0$. If $\gamma = 0$, then $\alpha + \beta \in K_0$ by (FR4). If $\gamma \notin K_0$, let $\kappa = \alpha$, $\mu = \epsilon\cdot \gamma$, $\lambda_1 = \epsilon $ and $\lambda_2 = \beta$. Then $\alpha + \beta \in K_0$ by (FR6$'')$. In the language of tracts, this means precisely that $\alpha + \beta \in N_{F_K}$.
\end{proof}

Combining Proposition~\ref{prop:FR6''impliesSF} with Theorem~\ref{thm:mainthm}, we recover the following special case\footnote{Dress and Wenzel prove in \cite[Theorem 2.7]{Dress-Wenzel92} that this remains true with ``simplified fuzzy ring'' replaced by ``weakly distributive fuzzy ring''.} of \cite[Theorem 2.7]{Dress-Wenzel92}:
\begin{thm*}
A simplified fuzzy ring which satisfies (FR6$''$) is perfect. 
% In particular, every ring and every weakly distributive fuzzy field or, more generally, integral domain is perfect.
\end{thm*}

The tract $\P'$ appearing in Example~\ref{ex:MSFbutnotSF} comes from 
% is associated, by the construction in \cite[Appendix B]{Baker-Bowler19}, to 
a simplified fuzzy ring $K$.
% \footnote{\mb Is this clear? \tz I think so} 
Since $\P'$ does not satisfy (SF), Proposition~\ref{prop:FR6''impliesSF} implies that $K$ does not satisfy (FR6$''$). On the other hand, Theorem~\ref{thm:strongmainthm} applies to $\P'$ since $\P'$ {\bf does} satisfy (MSF). This shows that Theorem~\ref{thm:strongmainthm} is strictly stronger than \cite[Theorem 2.7]{Dress-Wenzel92}, at least when we restrict the latter to simplified fuzzy rings. 

\begin{small}
\bibliographystyle{plain}
\bibliography{matroid}
\end{small}

\end{document}